\documentclass[11pt]{amsart}
\usepackage{amssymb}
\usepackage{latexsym}
\usepackage{amsmath}
\usepackage{amsthm}
\usepackage{graphicx}

\newtheorem{theorem}{Theorem}[section]

\newtheorem{prop}[theorem]{Proposition}

\newtheorem{cor}[theorem]{Corollary}
\newtheorem{lemma}[theorem]{Lemma}

\newtheorem{lemma-definition}[theorem]{Lemma-Definition}

\numberwithin{figure}{section}
\numberwithin{equation}{section}

\newtheorem{defn}[theorem]{Definition}

\newtheorem{rmk}[theorem]{Remark}

\newtheorem{example}[theorem]{Example}

\newcommand{\er}{{\Diamond}}
\newcommand{\less} {{\smallsetminus }}

\newcommand{\MS}{{\medskip}}
\newcommand{\NI}{{\noindent}}
\newcommand{\eps}{{\varepsilon}} 
\newcommand{\ind}{{\rm ind}}
\newcommand{\gr}{{\rm gr}}
\newcommand{\ov}{\overline}
\newcommand{\FS}{{\rm FS}}
\newcommand{\Si}{{\Sigma}}

\newcommand{\Ee}{{\mathcal E}}
\newcommand{\p}{{\partial}}
\newcommand{\C}{{\mathbb C}}

\newcommand{\R}{{\mathbb R}}
\newcommand{\N}{{\mathbb N}}

\newcommand{\Mm}{\mathcal{M}}
\newcommand{\oMm}{\ov{\mathcal{M}}}

\newcommand{\om}{\omega}
\newcommand{\al}{\alpha}
\newcommand{\be}{\beta}
\newcommand{\de}{\delta}
\newcommand{\ga}{\gamma}
\newcommand{\la}{\lambda}
\newcommand{\ka}{\kappa}

\newcommand{\se} {{\stackrel{s}\hookrightarrow}}
\renewcommand{\epsilon}{\varepsilon}

\begin{document}

\title{A remark on the stabilized symplectic embedding problem for ellipsoids}
 \author{Dusa McDuff}
  \thanks{partially supported by NSF grant DMS1308669}
\address{Department of Mathematics,
 Barnard College, Columbia University}
\email{dusa@math.columbia.edu}
\date{revised November 19, 2018}
\keywords{stabilized symplectic capacity, symplectic embeddings of ellipsoids, embedded contact homology}
\subjclass[2010]{53D05,57R17}

\begin{abstract}
This note constructs sharp obstructions for stabilized symplectic embeddings of an ellipsoid into a ball, in the case when the initial 
$4$-dimensional ellipsoid has \lq eccentricity' of the form $3\ell-1$ for some integer $\ell$.   This version is revised (after publication) to correct some of the examples.
\end{abstract}

\maketitle

\section{Preliminaries}

The question of when one symplectic ellipsoid embeds symplectically into another of the same dimension has turned out to be very fruitful, its answer giving a beautiful  illustration of the tension between rigidity and flexibility in symplectic geometry.  The situation in four dimensions is completely known (at least in principle), while there are many 
open questions in higher dimensions.  (For a survey, see Schlenk \cite{Schl}.) This brief note concerns the stabilized problem, in which one tries to use knowledge of the four-dimensional problem to understand symplectic embeddings $$
\Phi: \Ee\times \R^{2k}\;\;\se\;\;  \Ee'\times \R^{2k},
$$
 where $\Ee,\Ee'$ are four-dimensional and $k>0$.   It was discovered by Hind--Kerman and Cristofaro-Gardiner--Hind in \cite{HK,CGH} that some  four-dimensional embedding obstructions persist when stabilized: but not all --- for example,   the volume obstruction obviously disappears.    This note explains a simple new method for finding embedding obstructions, that gives sharp results 
for the case when $\Ee = E(1,3m-1)$ and $\Ee'$ is a ball.  

Identify Euclidean space  $\R^{2n}$ with $\C^n$  by taking
$$
(x_1,y_1,\cdots,x_n, y_n) \;\;\equiv\;\;(z_1,\cdots z_n),\;\; \mbox{ where } \;\; z_j = x_j+iy_j, 
$$
 and consider the standard symplectic form  $\om_0 = \sum_{j=1}^n dx_j\wedge dy_j$.
Every ellipsoid in Euclidean space is equivalent under a symplectic linear transformation to one in standard form
$$
E(a_1,\dots,a_n): = \Bigl\{z\ \big| \ \pi \sum_j \frac {|z_j|^2}{a_j} \le 1\Bigr\},\qquad 0<a_1<\dots< a_n.
$$
Thus $E(a,a)=: B^4(a)$ denotes the $4$-ball of \lq\lq capacity"\footnote
{
The capacity, being an area,  is a much more natural measure in symplectic geometry than the radius.}
 $a$, and hence radius $r$ where $a=\pi r^2$.
We are interested in calculating the function
$$
c_k(x) : = \inf \bigl\{ \mu : E(1,x)\times \R^{2k}\;\se\; B^4(\mu)\times \R^{2k}\bigr\},\quad x\ge 1, \ k\ge 0,
$$
where we write $U\se V$ if $U,V$ are subsets of $\R^{2n}$ and there is a symplectic embedding\footnote
{
A  symplectic embedding $f:U\to V$  is a smooth injective map that preserves the symplectic form, i.e. $f^*(\om_0) = \om_0$.} 
 from $U$ to $V$.
It is not hard to show that this function is continuous and nondecreasing.    
When $k=0$, it was calculated in \cite[Thm~1.1.2]{MS}: 
\begin{itemize}\item
$c_0(x)$  is piecewise linear when $x<\tau^4$ (where $\tau = \frac{1+\sqrt 5}2$ is the golden mean), with  
intervals on which  $c_0$ has the form $x\mapsto \la x$ alternating with  intervals on which it is constant, and where the numerics (e.g. slopes) are governed by the odd index Fibonacci numbers $f_{2i+1}, i\ge 0$; 
 \item in particular  the left endpoints of the constant intervals lie at $x=\frac{f_{2i+5}}{f_{2i+1}}$, and the constant values are  
 $\frac{f_{2i+5}}{f_{2i+3}}$ for all $i\ge0$;
 \item $c_0(x)$  equals the volume constraint $vol(x) =\sqrt x$ for $x>(\frac{17}6)^2$ --- notice that if $E(1,x)\se B^4(\mu)$ then the volume $\frac x2$ of $E(1,x)$ cannot be more that the volume   $\frac {\mu^2}2$ of the ball;
 \item the interval $\tau^4 < x < (\frac{17}6)^2$ is a transitional region with a finite number of intervals on which $c_0$ is linear, complemented by intervals on which it equals $\sqrt x$.
\end{itemize}
The part of this graph  over $[1,\tau^4]$ is called the Fibonacci stairs, and it was proven in \cite{CGH} that this stabilizes; in other words
\begin{align}\label{eq:lesstau}
c_k(x) = c_0(x),\quad 1\le x\le \tau^4.
\end{align}
In particular, because 
$3f_{k+3} = f_{k+1} + f_{k+5}$ for all $k$,
we find that
$$
c_k(x) = \frac{3x}{x+1} \quad \mbox{when } x = \frac{f_{2i+5}}{f_{2i+1}}< \tau^4.
$$
Due to an explicit  folding construction of Hind~\cite{Hi}, we also know that 
\begin{align}\label{eq:conj}
& c_k(x) \le \frac{3x}{x+1},\quad \mbox{ when }\; x> \tau^4,\,\, k>0.
\end{align}

As is shown by Remark~\ref{rmk:autreg}~(ii),  this function $x\mapsto  \frac{3x}{x+1}$ has a  geometric connection to the embedding problem via the Fredholm index.\footnote
{
See \cite{CGHM} for further discussion of the significance of this function.}  
Further, its graph crosses the volume obstruction $x\mapsto \sqrt x$ precisely at the point $x=\tau^4$.  Hence
it is natural to conjecture that \eqref{eq:conj} should be an equality. To this end, we know from \cite{CGHM} that
\begin{align}\label{eq:CGHM}
&c_k(x) = \frac{3x}{x+1} \quad \mbox{when }\; x = \frac{f_{4i+6}}{f_{4i+2}},\quad i\ge 0,
\end{align}
i.e. at the points $\frac 81, \frac{55}8, \frac{377}{55}$ and so on.   Further, by
building on the results in \cite{CGH,CGHM}, Hind--Kerman show in \cite{HK,HK2} that 
\begin{align*}
c_k(x) = \frac{3x}{x+1} = \frac{3f_m-1}{f_m},\quad  \mbox{ if }\; x = 3f_m-1 \mbox{ where } m\not\equiv_4 2.
\end{align*}

In this note we extend this result to the case when $x$ is an arbitrary integer of the form $3m-1$.

\begin{theorem}\label{thm:1} $c_k(3m-1) = \frac {3m-1}m$ for all integers $m>0$ and all $k\ge 1$.
\end{theorem}

Note that the first two values, $c_k(2), c_k(5)$ are known by \eqref{eq:lesstau}, while $c_k(8)$ is known by the case $i=0$
of \eqref{eq:CGHM}.    The argument given here
uses knowledge of the cases $m=1,2$ to calculate $c_k(3m-1)$ for $m\ge 3$.

\MS

\NI{\bf Note:}  This paper is revised to correct the details of some of the examples: see Remarks~\ref{rmk:trivcyl}, ~\ref{rmk:glu} and Example~\ref{ex:1}.  The published proof of the main theorem is correct.
\MS

\NI {\bf The stabilization theorem}

The proofs of these results about  $c_k, k>0,$  rely on the following stabilization result,
which was proved in \cite[Proposition~3.6.6]{CGHM} using arguments originating in \cite{HK,CGH}.  Before stating it, we describe the basic geometry. (For more information, see for example \cite[\S2.2,\S3.1]{CGHM}.)

We denote by $\C P^2(\mu)$ the projective plane equipped with the standard (Fubini--Study) symplectic form $\om_{\FS}$, scaled so that the line has area $\mu$.  The manifold  $\C P^2(\mu)$ may be obtained from the ball $B^4(\mu)\subset (\R^4, \om_0)$ by identifying its boundary to a (complex) line via the Hopf map.  We assume $\mu$ chosen so that there is a symplectic embedding  $\Phi:E(1,x)\to {\rm int\,} B^4(\mu)\subset \C P^2(\mu)$, and
denote by $(\ov{X}_x, \om)$ the negative completion\footnote
{
An open symplectic manifold is said to be complete if its ends are either positive (modelled on $Y \times [0,\infty)$  with symplectic form $d(e^{s}\al)$ where $(Y,\al)$ is some contact manifold and $s$ is the coordinate on $\R$) or negative (modelled on  $Y \times (-\infty,0]$  with symplectic form $d(e^{s}\al)$).  Thus
 positive ends are convex (expanding towards $+\infty$), while negative ends are concave (contracting as one moves towards $-\infty$). If $(Y,\al)$ is a contact hypersurface in $(X,\om)$ the process of {\it stretching the neck} elongates the collar around $Y$ so that the manifold splits into two pieces $X^+, X^-$.  The converse gluing process attaches the (truncated) negative end of $X^+$ to the (truncated) positive end of $X^-$ as in Proposition~\ref{prop:HT} below. }
 of $\C P^2(\mu)\less \Phi(E(1,x))$.  Thus we remove the interior of 
$\Phi(E(1,x))$ and replace it by a copy of $\p E(1,x) \times (-\infty,0]$  with symplectic form $d(e^{s}\al)$, where $s$ is the coordinate on $(-\infty,0]$ and the $1$-form $\al $ on $\p E(1,x)$  is chosen so that  $\Phi$ pushes $d\al$ forward to the restriction $\om_{\FS}|_{\Phi(\p E(1,x)}$.   Thus $\al$ is a contact form on $\p E(1,x)$ whose Reeb vector field $R_\al$ is a multiple of the symplectic gradient\footnote
{
The symplectic gradient vector field $X_H$ is given by the identity $\om(X_H,\cdot) = dH$, and is tangent to the kernel of $\om|_{H=const.}.$}
 of 
 the defining function
 $H: = \frac \pi 2 (|z_1|^2 + \frac{|z_2|^2}{x})$ of the ellipsoid  boundary $\p E(1,x)$.  In particular, when $x>1$ is irrational, the Reeb flow has precisely two  simple closed orbits, the short orbit $\be_1(x)$ round the circle $z_2=0$ in $\p E(1,x)$ that bounds a disc of area $1$ and the long orbit $\be_2(x)$ round the circle $z_1=0$ in $\p E(1,x)$ that bounds a disc of area $x$.   

We consider almost complex structures $J$ on 
 $\ov{X}_x$ that are 
 \begin{itemize}\item
 tamed by $\om$ (i.e. $\om(v,Jv)>0$ for all nonzero tangent vectors $v$), and also 
 \item compatible with the negative end, i.e. are invariant under translation by $s\in \R$ and also 
 have the property that $R_\al = J(\frac{\p}{\p s})$ where $R_\al$ is the Reeb vector field on $\p E(1,x)$ as above. 
 \end{itemize}

It was discovered by Hofer~\cite{Ho} that there is a good theory of finite energy $J$-holomorphic curves in manifolds such as $\ov{X}_x$.  
These curves\footnote
{
Strictly speaking, by \lq\lq curve" we mean an equivalence class of maps $u:(\dot\Si,j)\to ( \ov X_\mu, J)$, where we quotient out by 
the finite dimensional group of biholomorphisms of the domain. In the situations considered here, this group is  trivial whenever 
the total number of  punctures or marked points on the domain is  at least three.}
 $C$ are given by proper maps
$$
u: (\dot\Si,j)\to ( \ov X_\mu, J),\qquad du\circ j = J(u) \circ du,
$$
where $(\dot\Si,j)$ is a punctured Riemann surface, that satisfy a finite energy condition.  This condition  is designed so that the image of the restriction of $u$ to a sufficiently small neighbourhood of a puncture is a cylinder in the negative end  $\p E(1,x) \times (-\infty,0]$ that converges exponentially fast to a multiple cover 
of one of the periodic orbits  of the Reeb flow on $\p E(1,x)$.
Further, one can show that these maps $u$
are solutions of a Fredholm operator  that is defined on a suitable space of maps $u: \dot\Si\to \ov{X}_x$.  In the situation at hand, 
the Fredholm index\footnote
{
See \cite[\S2]{CGHM} for more details of these calculations.
Further, the degree of $C$  is the number of times it intersects  the line at infinity.} 
 of a genus zero curve $C$, with degree
 $m$  and $n_1 + n_2$ negative ends is
$$
\tfrac 12 \ind(C) = -1 + 3m - \sum_{i=1}^{n_1}\Big(t_i + \Big\lfloor \frac{t_i}x\Big\rfloor\Big) - \sum _{j=1}^{n_2}\Big(r_j + \Big\lfloor r _jx\Big\rfloor\Big),
$$
where  $t_1,\dots,t_{n_1}$ are the multiplicities of the $n_1$ ends on $\be_1(x)$  and $r_1,\dots,r_{n_2}$ are those of the $n_2$ ends on $\be_2(x)$.
Since $J$ is tamed by $\om$, $\om$ pulls back to  an area form on $\dot\Si$. 
By ignoring contributions to $\int_{\dot \Si} u^*\om$ from the trivial (cylindrical) directions $\frac{\p}{\p s}, J(\frac{\p}{\p s}) = R_\al$ in the cylindrical end of $\ov X_x$, one can define the notion of the {\it action} of $C$.  This is finite and positive for finite energy curves, and in the case at hand equals
$m\, \mu - \sum t_i - \sum_j xr_j $.   Hence, if such a curve $C$ exists and persists under perturbations, we obtain  the lower bound
$\mu \ge \frac 1m\bigl(\sum t_i + \sum_j xr_j\bigr)$ for the size of a ball that contains a copy of $E(1,x)$.  
Many obstructions to symplectic embeddings come from the existence of such curves.

\begin{defn}  
If $x$ is irrational  and $3m = t + \lceil \frac tx \rceil$, we define $\Mm\bigl(m,x,t\bigr)$ to be
the moduli space of genus zero curves in $\ov X_x$, of degree $m$, and  with one negative end of multiplicity $t$ on the short orbit $\be_1(x)$ on the ellipsoid.  
  Its elements have Fredholm  index 
$$
\ind(C): = 2\Big(-1 + 3m - t - \Big\lfloor\frac t{x}\Big\rfloor\Big) =2\Big(3m - t -\Big\lceil \frac t{x}\Big\rceil\Big) = 0.
$$
\end{defn}

\begin{rmk}\label{rmk:autreg}\rm  (i) By the principle of \lq\lq automatic regularity" for genus zero curves in $4$-manifolds (see \cite{W}), the number of curves in 
 $\Mm(m, x, 3m-1)$ is constant 
as $\om$ and $ J$  are varied, and they all count positively.  In particular, curves in $\Mm(m, x, 3m-1)$ do not disappear as  the embedding $\Phi: E(1,x)\to B^4(\mu)$ and the parameter $\mu$ change.
\MS

\NI (ii)  If $x = \frac pq + \eps$ and $t=p$ then $1 + \lfloor\frac t{x}\rfloor = \lceil \frac t{x}\rceil = q$ for small $\eps>0$.  Hence in this case the index condition is $3m=p+q$, which gives the lower bound  $\mu\ge \frac pm = \frac{3p}{p+q}$.
\hfill$\er$
\end{rmk}

Here is the key stabilization result from \cite{CGHM,HK,CGH}.
 The basic reason why it holds is that the condition that the curve $C\in \Mm(m,x,t)$ has genus zero and one negative end means that its Fredholm index does not change when it is embedded in the product $\ov X_x\times \R^{2k}$.  
 Hence it  persists under stabilization, and  one can show that  it also persists as the embedding of the ellipsoid is 
 deformed.  As explained above, the embedding obstruction comes from the fact that its action  must be positive.    
  
\begin{prop}\label{prop:stab}{\bf (Stabilization)}  Suppose that the index condition
$3m = t + \lceil \frac tx \rceil$ holds, and that
 $\Mm(m,x,t)$ is nonempty.  Suppose further 
that there is no decomposition $m= \sum_{i=1}^n m_i, t = \sum _{i=1}^n t_i$ with $n>1$, $m_i, t_i>0, \forall i$, and such that
\begin{align}\label{eq:decomp} 3m_i = t_i +  \Big\lceil\frac{t_i}x\Big\rceil\qquad \forall \ i.
\end{align} 
Then $c_k(x) \ge \frac tm.
$
\end{prop}

\begin{cor}\label{cor:1}  If $\Mm(m,3m-1+\eps,3m-1)$ is nonempty for all $\eps>0$, then 
$$
c_k(3m-1) = \frac {3m-1}m = \frac{3x}{1+x}, \qquad x: = 3m-1.
$$
\end{cor}
\begin{proof}   For any decomposition of the form
$$
m= \sum_{i=1}^n m_i, \quad t= 3m-1 = \sum_{i=1}^nt_i, \mbox{ where }  \; m_i> 0,\;\; n>1,
$$ 
we have  $ \lceil \frac{t_i}{x}\rceil = 1$,
so that  $3m_i <  t_i + \lceil \frac{t_i}x\rceil $ for at least one $i$.  Hence there is no  decomposition of the form~\eqref{eq:decomp}, so that  Proposition~\ref{prop:stab} implies that $c_k(x)\ge \frac{3m-1}{m}$ for all irrational $x>3m-1$.  But  $c_k(x)$ is continuous.  (As explained in \cite{MS} this holds because $E(1,x)\subset E(1,\la x) \subset \la E(1,x)$  for all $\la \ge 1$.)  Moreover, Hind shows in \cite{Hi} that
$c_k(x)\le \frac{3x}{x+1}$ for all $x$.   Hence we must have equality at the integer $x=3m-1$.
\end{proof}

\begin{rmk}\label{rmk:stab}\rm  The conditions in Proposition~\ref{prop:stab} hold for the triples $$
(m,x,t) = \Big(\frac{p+q}3, \frac pq + \eps, p\Big)
$$ provided that $\gcd(p,q) = 1$ and $\eps>0$ is sufficiently small.  Indeed in this case 
$\frac px$ is just less than $q$, while  if $t_i<p$ the difference $ \lceil \frac{t_i}x\rceil - \frac {t_i}x$ is at least $\frac 1p$. Similarly, it
is not hard to check that if  $\gcd(3,p+q) = 1 = \gcd(p,q)$, then triples of the form $(p+q,\frac pq+ \eps, 3p)$ such as $(8=7+1, 7+\eps, 21)$ also satisfy these conditions.  However,  if 
 $\gcd(p,q) \ne 1,3$ there is $n>1$ such that $m_i = \frac{p+q}{3n}, t_i = \frac pn$ are integers, so that there is a decomposition as in \eqref{eq:decomp} with $n$ equal terms.
 
 \end{rmk}

\section{Constructing curves}

We now show how to construct curves $C$ that satisfy the conditions of Proposition~\ref{prop:stab}. 
There are different ways to do this.  In \cite{HK,CGHM}, the idea was to start with known spheres in a suitable blow up of $\C P^2$ and then \lq\lq stretch the neck" around the boundary of the ellipsoid, while \cite{CGH} worked directly in $\ov X_x$ and used results from embedded contact homology (ECH).  
The short note \cite{HK2} takes the genus zero curves constructed in \cite{CGH,CGHM} and attaches cylinders to obtain obstructions at the numbers $3f_k-1$, where $f_k$ is the $k$th Fibonacci number for $k$  not divisible by $4$.   In this note,  we attach cylinders as in \cite{HK2} to curves
that are constructed using the technique of obstruction bundle gluing that was developed by Hutchings--Taubes in \cite{HT} during the course of their work on ECH.  
\MS

\NI {\bf Obstruction bundle gluing}

   In the following lemma we use $y$ as the parameter for the ellipsoid, since in the application we will be gluing at $\p E(1,y)$ to obtain a curve with negative end on $\p E(1,x)$.
For short,  we will denote by $\be_1(y)^s$ (or $\be_1^s$) the $s$-fold cover of the short orbit on $\p E(1,y)$. 

\begin{lemma}\label{le:2}   Let $(p_i): = (p_1,\dots,p_r)$ be a partition of $s: = \sum p_i$ with $p_1>p_2>\dots > p_r$ such that for some irrational $y>1$ we have
$\sum_i \lceil \frac{p_i}y\rceil =   \lceil \frac sy \rceil$.  Then there is a genus $0$ and index $0$  trajectory in the symplectization $\p E(1,y)\times \R$ with
$r$ positive ends on $\be_1$ of multiplicities $p_1,\dots, p_r$ and  one negative end  asymptotic to $\be_1^s$.  
\end{lemma}
\begin{proof}  There is an $s$-fold branched cover $C_{neck}$ of the constant cylinder $\be_1\times \R$ with ends of the appropriate multiplicities.  It has (half) Fredholm index
$$
\tfrac 12\ind(C_{neck}) = -1 + r + \sum _i p_i + \sum_i \Big\lfloor \frac{p_i}y\Big\rfloor  - s - \Big\lfloor \frac sy\Big\rfloor
= \sum_i \Big\lceil \frac{p_i}y\Big\rceil -   \Big\lceil \frac sy \Big\rceil = 0,
$$
since $ \lceil \frac sy \rceil = 1 + \lfloor \frac sy\rfloor$.
\end{proof}

 Although the above trajectories have index $0$, the moduli space of these trajectories is in general NOT cut out transversally. So they are not \lq\lq regular", i.e. they belong to a moduli space $\Mm$ of too high a dimension because there are too many branched $s$-fold coverings  $S^2\to S^2$ that satisfy the relevant branching conditions. On the other hand, the cokernels of the linearized Cauchy--Riemann operator at the points of $\Mm$ have constant dimension  and hence form a bundle over $\Mm$, called the obstruction bundle because it obstructs deformations of the elements in $\Mm$.
In the process of gluing, one first constructs a \lq\lq preglued" curve which is approximately holomorphic and then uses a Newton process to obtain a holomorphic curve.   In the regular case there is typically a one-to-one correspondence between the preglued objects and the resulting holomorphic objects.
In \lq\lq obstruction bundle gluing"  one shows that the number of  holomorphic curves obtained by the Newton process may be calculated from properties of the obstruction bundle. (In the simplest case it will just be its Euler class.)

The following result is a cornerstone of ECH.  
\MS

\begin{prop}\label{prop:HT}  Suppose given numbers $(p_i), s$ as in Lemma~\ref{le:2}, and let $\be_1$ be the short orbit on $\p E(1,y)$.  Let $\ov V, \ov Y$ be complete symplectic manifolds such that $\ov V$ has negative
 end $\p E(1,y)$ while $\ov Y$ has positive end $\p E(1,y)$; denote by $\ov X$ the manifold obtained by attaching $\ov V$ to $\ov Y$ along their common end.    Suppose that 
\begin{itemize}\item[-] 
 $u_i$ and $u $ are  somewhere injective index $0$ curves,  where $u_i$ in $\ov V$ has bottom end on $\be_1^{p_i}$   and $u$ in $\ov Y$ has top end on $\be_1^s$.
 \end{itemize}   
 Then there is a positive number $c_\ga$ of ways to glue the trajectory $(\bigsqcup u_i)$ to $u$ to form an index $0$  curve $v$ in $\ov X$.  
\end{prop}
\begin{proof}  
This follows from \cite[Thm~1.13]{HT}.  In that paper the authors prove an analogous result about the  gluing of  index $1$ trajectories in the symplectization
$\p E(1,y)\times \R$ when it is equipped with an $\R$-invariant almost complex structure.  Because the glung operation is local,  i.e. it depends only on what is happening near the neck region, their arguments apply equally well in the current  situation; in particular, somewhere injective curves are embedded near the neck and we only consider admissible almost complex structures that, by definition, are  $\R$-invariant near the neck.  However, because there is no global $\R$-action, we consider index $0$ rather than index $1$ trajectories.  Notice that 
our hypothesis on the trajectories $u_i, u$ and numbers $p_i,s$ imply that  the pair $(u_+: = (u_i), u_-: = u)$ is a 
gluing pair in the sense of \cite[Defn.~1.9]{HT}.  It is important that the $u_i$ are distinct,  but this follows from the condition  $p_1>p_2>\dots > p_r$ on the partition.
  Therefore it remains to check that  the gluing coefficient (called $c_\ga(S)$ in \cite{HT}) is positive.  
in  the current situation, we may directly apply  \cite[Defn.~1.22]{HT} (since $\ka_\theta(S) = 1$ and $k'=\ell'=0$) so that $c_\ga(S)$ is given by formula (1.14)  with $\ell=1$.  But in this case the positivity of $c_\ga(S)$ is asserted in 
 \cite[Remark~1.21]{HT}.
\end{proof}

\begin{example}\label{ex:257}\rm  Consider the decomposition $(p_1,p_2)=(2,5)$ of $s=7$.   With $5< y < 7$ we have
$\sum_i \lceil \frac{p_i}y\rceil =  1 + 1 = \lceil \frac sy \rceil$, so that two curves with bottom end multiplicities $2,5$ may be glued to one with top  end of multiplicity $7$.  In this case, the recipe in \cite[(1.14)]{HT} gives the gluing coefficient\footnote
{
The definition of $f_{1/y}(5,2\,|\,7)$ in \cite[Def~1.20]{HT} is too long to reproduce here; we mention it to make it easier for the interested reader to check 
what is needed from  \cite{HT}.
}
$$
c_\ga(S) = f_{1/y}(5,2\,|\,7) =  f_{1/y}(2|2)\ \de_{1/y}(5,7) =\de_{1/y}(2,2) \ \de_{1/y}(5,7) = 2\cdot 2 = 4,
$$
where $\de_\theta (a,b) = b\lceil a\theta\rceil - a\lfloor b\theta\rfloor$ so that $\de_{1/y}(2,2) = 2\lceil \frac 2y\rceil - 2\lfloor \frac 2y\rfloor = 2$.
\hfill$\er$
\end{example}

\begin{rmk}\label{rmk:trivcyl}\rm  The gluing theorem \cite[Thm~1.13]{HT}  also  applies if the $u_i$ are somewhere injective and distinct, but  $u$ is a trivial cylinder in a symplectization.\footnote{
By definition, a {\it trivial cylinder}  in a symplectization $\Si\times \R$ is
an unbranched  multiple cover of the  cylinder $\be\times \R$ over an asymptotic Reeb orbit $\be$ in $\Si$.
One might be able to weaken this requirement, but to do this would require detailed analysis of the gluing argument.  The point is that in this case the argument is not local to the top end of $u$, but because of its invariance under the $\R$-action,  involves a neighborhood of the whole curve.}    
In this case condition (d) in \cite[Defn.~1.9]{HT} requires that the partition $(s)$ be maximal in the partial order defined in 
\cite[Defn.~1.8]{HT}.   The calculations in \cite[Appendix~A]{Hlect} relevant to \cite[Exercises~3.13,3.14]{Hlect} show that the  negative ECH partition\footnote
 {
 Here  $p^-_{\theta}(k)$ (resp. $p^+_{\theta}(k)$) denotes the ECH partition at the negative (resp. positive) end of a trajectory with total multiplicity $k$ around an orbit with monodromy angle $\theta$.} 
$p_{1/y}^-(s)$ is always $\ge (s)$.  Hence 
 this is equivalent to  requiring that 
 the negative ECH partition    $p_{1/y}^-(s)$
at $\be_1^s(1/y)$  has the single term $(s)$.\footnote
 {
 In the published version of this paper the condition in \cite{HT} was misinterpreted 
 as the requirement that $ p_{1/y}^+(s) = (s)$, a condition that  is sometimes satisfied when the partition $(p_i)$ is nontrivial.  This led to the assertion that certain moduli spaces $\Mm(m, x, t)$ are nonempty in cases when one can show that in fact  they have to be empty.  We have corrected the examples discussed in the remarks below to take this into account: see Example~\ref{ex:1} and
Remark~\ref{rmk:glu}. }
  But it is easy to check that this condition is never satisfied if 
there also is a partition  $p=(p_i)$ that satisfies  $\sum_i \lceil \frac{p_i}y\rceil =   \lceil \frac sy \rceil$ and has more than one entry, since the corresponding piecewise linear path lies between the lines of slope $1/y$ and $1/s$. (For relevant background see \cite
{Hlect} or \cite[\S2]{CGHM}.)
\hfill$\er$
\end{rmk}

\MS

\NI {\bf Cylinders}

Given irrational numbers $y \le  x$, let 
$\ov Y$ be the cobordism from $\p E(1,y)$ (the positive end) to $\p E(1,x)$  (the negative end) obtained by completing the region
$E(1,y) \less E(\la,\la x)$ at its positive and negative ends, where $\la>0$ is sufficiently small\footnote
{
The precise value of $\la$ is irrelevant, since all the curves we construct here persist under deformations.}
 that $E(\la,\la x)\subset E(1,y)$.
Let 
$\Mm_{cyl}\bigl(\be_1^s(y), \be_1^t(x)\bigr)$ be the moduli space of $J$-holomorphic cylinders in $\ov Y$ with positive end on $\be_1^s(y)$ and negative end on $\be_1^t(x)$, where $s,t\in \N$.  We  assume that these cylinders have  index zero, i.e. that
\begin{align}\label{eqn:ind} 
 \tfrac 12\ind_{cyl}(\be_1^s(y)) : = s + \lfloor \frac sy\rfloor =  \tfrac 12 \ind_{cyl}(\be_1^t(x)) : = t + \lfloor \frac tx\rfloor.
\end{align}
Since $y<x$ this condition implies that $s\le t$. 

Our general method is to try to construct curves with a single negative end in $\ov X_x$  by gluing several curves in $\ov X_y$ with end multiplicities $p_1,\dots, p_n$ where $\sum p_i = s$ to an index zero cylinder in $\Mm_{cyl}\bigl(\be_1^s(y), \be_1^t(x)\bigr)$.
 Note that if $s=t$ then there is such a cylinder, namely the $s$-fold cover of the cylinder
 from $\be_1(y)$ to $\be_1(x)$ that one can show must always exist.  However, this cylinder is obviously not somewhere injective and hence, as we saw in Remark~\ref{rmk:trivcyl},   Proposition~\ref{prop:HT} does not apply.
On the other hand if $\frac t2< s<t$  then (if it exists) the cylinder is somewhere injective since the multiplicities of its ends are mutually prime, and hence, since we assume it has index zero, must be regular.
 Hence 
the standard gluing argument shows that   
  it can be glued to any regular  index zero curve at its top or bottom end.   
 \MS

Because we are dealing with cylinders, there are several relevant  homology theories, such as 
cylindrical contact homology and equivariant symplectic homology.  In such theories, ellipsoids have one generator in each even degree, so that  the cobordism maps in these theories ideally would assert that the desired cylindrical trajectories exist.   
 However, because of transversality issues, the situation is very delicate. Equivariant symplectic homology (cf. Gutt~\cite{Gu} or Ekholm--Oancea~\cite{EO}) concerns cylindrical trajectories whose form is a little different  from those considered here since they satisfy an equation perturbed by a Hamiltonian function, while Pardon~\cite{P} has developed a very general version of contact homology that can be specialized to the case of interest here.
A more elementary
 version of cylindrical contact homology is still under construction (see Hutchings--Nelson~\cite{HN}).  Note that,  even though the basic chain complex in cylindrical contact homology can in certain cases be defined using  $J$-holomorphic curves of the kind considered above, when constructing cobordisms  one must perturb the equation in some way.  For example, \cite{HN} uses domain dependent almost complex structures, i.e.  $J$ that depend on the $S^1$ coordinate in the domain $S^1\times \R$ of the cylinder.  
As such $J$ converges to one that is domain independent, the compactness theorem of \cite{BEHWZ} asserts that the cylinder 
converges to a $J$-holomorphic building from $\be_1^s(y)$ to $ \be_1^t(x)$.  For example, it could be a series of cylinders stacked end to end; see Example~\ref{ex:cyl}  below  in which we give examples where $\Mm_{cyl}\bigl(\be_1^s(y), \be_1^t(x)\bigr)$ must be empty even though the index condition is satisfied.   On the other hand, one may conclude from these general theories that there always is a building with top end 
$\be_1^s(y)$ and bottom end $\be_1^t(x)$.      In the absence of a readily quotable proof of this result, Hind--Kerman~\cite{HK3} give an independent more direct argument.

Here we shall assume the above result, and simply sketch the proof of the following lemma, which is a version of a result in 
 \cite{HK2,HK3}. Note that by Example~\ref{ex:1} one cannot always a find a cylinder with end multiplicities $t-1,t$ even if the index condition holds; some  conditions  on $y,x$ are  necessary.

\begin{lemma}\label{le:1}  Suppose that $\ind_{cyl}\bigl(\be_1^{t-1}(y)\bigr) = \ind_{cyl}\bigl(\be_1^t(x)\bigr)$, where $ y<t < x$. 
Then $\Mm_{cyl}\bigl(\be_1^{t-1}(y), \be_1^t(x)\bigr)$ contains a regular, somewhere injective $J$-holomorphic cylinder of index $0$.  In particular,
there is such a cylinder in $$\Mm_{cyl}\bigl(\be_1^{3m-2}(y), \be_1^{3m-1}(x)\bigr)$$
 if $y<3m-2$ and $ x>3m-1$. 
\end{lemma}
\begin{proof}[Sketch of proof]  Suppose that for some generic $J$ there is a cylindrical building from $\be_1^{t-1}(y)$ to $\be_1^{t}(x)$ with at least two levels.  
The building has one top end, one bottom end and genus zero, 
and one can choose the almost complex structure on the cobordism part so that the maximum principle  implies that every component has at least one positive end.  Hence, because the whole curve has genus $0$, 
each connected component in any level of the building must have a single positive end.  

We claim the building must contain a component of positive index. To see this, note first that
 if any component has more than one negative end, 
the fact that the total curve has genus zero means  that all but one of these  components must be capped off by a union of components that together form a plane, i.e. with just one (necessarily positive) end.  One can check that these planes have positive index.    
Therefore, it suffices to consider the case when each component has just one negative end, and so is a cylinder. 
But, as in  \cite[Lemma~3.12]{HK}, Stokes' theorem implies that every cylinder in the symplectization $\p E(1,a)\times \R$ of an ellipsoid has nonnegative index, and   the index is zero only if the cylinder is  trivial.
  Therefore, because by definition no symplectization level of the building can consist only of 
trivial cylinders,
 the building must contain a component of positive index.  
 
 Hence it 
 also contains a component $C'$ of negative index, which must lie in   the cobordism level.
 (Every building has just one cobordism level, while it might have several levels in the top and bottom symplectizations.) 
Since for generic $J$ all somewhere injective curves have nonnegative index, 
the component  $C'$ must be 
the $n$-fold cover   $nC$
of a  curve $C$ with nonnegative index where $n>1$.   We now show that our numerical hypotheses imply that such a curve $C$ cannot exist.

Since the components of the building in the symplectization levels have nonnegative index,
 the top end of $nC$ has index at most that of $\be_1^{t-1}(y)$ while its bottom  has index at least that of $\be_1^{t}(x)$.
In particular,
because 
 $$
\tfrac 12 \ind (\be_2(x)) = 1 + \lfloor x\rfloor\ge 1+t> t=\tfrac12 \ind(\be_1^t(x))  = \tfrac12 \ind(\be_1^{t-1}(y)),
$$
$C$ cannot have a negative end on any multiple of $\be_2(x)$.  Hence its negative ends cover $\be_1(x)$ with some multiplicity $q$, where 
$q< t< x$.  Thus  the $\frac 12$-index of these ends  is $q$, no matter what the partition is.

Further the fact that $\ind(\be_1^t(x))  =  \ind(\be_1^{t-1}(y))$ implies that
$\lfloor \frac {t-1}y \rfloor = 1$, so that $y> \frac{t-1}2$.  In particular,  $$
\tfrac 12\ind(\be_2^{2}(y)) = 2 + \lfloor 2y \rfloor \ge t+1> 
\tfrac 12 \ind(\be_1^{t-1}(y)).
$$
Therefore,
the positive end of $C$  lies on $\be_1^p(y)$ where   $np\le t-1< t\le nq$.  Hence $p<q$, so that    
to obtain  $\ind(C) \ge 0$ we need  $\lfloor \frac py \rfloor \ge 1$.   Therefore we  must have
$y<p \le \frac{t-1}n$.  But we saw above that $y> \frac{t-1}2$.  Hence $n=1$, contrary to hypothesis.

Thus proves the first claim.  The second follows immediately because the conditions on $x,y$ imply that
$\tfrac 12 \ind(\be_1^{3m-2}(y)) = 3m-1 = \tfrac 12 \ind(\be_1^{3m-1}(x))$.
\end{proof}

\NI {\bf The Construction}

We are now in a position to construct the desired curves.   

\begin{lemma}\label{le:3k-1}  For all $m\ge 1$ and $x> 3m-1$ there is a  genus $0$ trajectory $C_m(x)$ with one negative end in  
$\Mm\bigl(m, x, 3m-1\bigr)$.
\end{lemma}
\begin{proof}  We prove this by induction on $m$.   The cases $m=1,2$ are proved in \cite{CGH}, which constructs the curve as an ECH (embedded contact homology) trajectory.  Alternatively, one
could blow up inside the ellipsoid in $\C P^2(\mu)$ and then look to see what happens to  the exceptional sphere in class\footnote
{
Here we denote the line class by $L$, and write $E_{1\dots k}$ as shorthand for the sum $\sum_{i=1}^k E_i$ where $E_i$ is the class of the $i$th exceptional divisor.}
  $L-E_{12}$ (when $m=1$) or $2L-E_{1\dots 5}$  (when $m=2$) when the  neck around  the boundary of the ellipsoid $\p E(1,x)$ is stretched.  In the limit the sphere breaks into at least two pieces, the top piece (called $C_U$) that lies in the completion $\ov X_x$ of $\C P^2(\mu)\less \Phi(\p E(1,x))$, the bottom piece that lies in the positively completed blown up ellipsoid, and possibly some other components lying in the neck region $\p E(1,x)\times \R$.
It turns out that the top of the resulting curve has the right structure.    (For details of this approach, see  \cite[Remark~3.1.8~(ii)]{CGHM}.)
This argument does not work for $m\ge 3$ because, as explained in  \cite[Remark~3.5.4~(i)]{CGHM} although there is a suitable exceptional class, when the neck is stretched the resulting curve will have more than one negative end.

Suppose now that for some $m\ge 2$ and $y\in (3m-1,3m+2)$ the curve $C_m(y)$ has been constructed, and consider its union with the trajectory $C_1\in 
\Mm\bigl(1, y,2)\bigr)$.  Then the partition $(3m-1,2)$ of $s = 3m+1$ satisfies the conditions in Lemma~\ref{le:2}
for $y \in (3m-1, 3m+1)$, so that there is a regular genus $0$ curve $C_{neck}(y)$ with two positive ends of multiplicities  
$(3k-1,2)$ and one negative end.  
  Further 
  \begin{align*}
&\tfrac 12 \ind_{cyl}\bigl(\be_1^{3m+1}(y)\bigr) = 3m+2  = \tfrac 12 \ind_{cyl}\bigl(\be_1^{3m+2}(x)\bigr), \\
& \qquad \qquad \qquad  \mbox{ if } \;3m-1< y< 3m+1, \;\mbox{ and }\;  x> 3m+2.
\end{align*}
Hence by Lemma \ref{le:1} there is a cylinder from $(y, 3m+1)$ to $(x, 3m+2)$, 
which has to be somewhere injective since the end multiplicities are relatively prime. 
Hence we may obtain the desired trajectory in $\oMm\bigl(m+1,\, x, \, 3m+2)\bigr)$ by gluing these three components at the neck  using Proposition~\ref{prop:HT}.
\end{proof}

\begin{proof} [Proof of Theorem~\ref{thm:1}]
This follows by combining the statement of Lemma~\ref{le:3k-1} for $x = 3m-1 + \eps$ with Corollary~\ref{cor:1}.
\end{proof}

\begin{example}\label{ex:cyl}\rm  Since the curves in $\Mm(1,x,2)$ and $\Mm(2,x,5)$ are embedded and turn out to be detected using ECH, one can show that there is exactly one element in each of these moduli spaces.  Therefore, if we also assume that there is one cylinder  
from $\be_1^7(7-\eps)$ to $\be_1^8(8+\eps)$, the calculation  in Example~\ref{ex:257} shows that our method constructs $4$ elements in $\Mm(3,8+\eps, 8)$.   
In \cite[\S3.5]{CGHM} we used a different method to find elements in this moduli space, and also gave plausibility arguments that there should be exactly $4$ such elements.   Thus, by further analysis of gluing operations, one might be able to show that $\Mm(3,8+\eps, 8)$  does have exactly $4$ elements. Further, it is not hard to check that at each stage  the relevant gluing coefficient is $4$.  Hence
our method  seems to construct  $4^{m-2}$ distinct elements in 
$\Mm(m,3m-1+\eps, 3m-1)$ for each $m\ge 3$.\hfill$\er$
 \end{example}
 
It does not seem easy to extend   the above argument to other values of $x$, mostly because it is hard to find suitable cylinders.  
 We end by discussing some cautionary  examples.

  \begin{figure}[htbp] 
     \centering
     \includegraphics[width=2.5in]{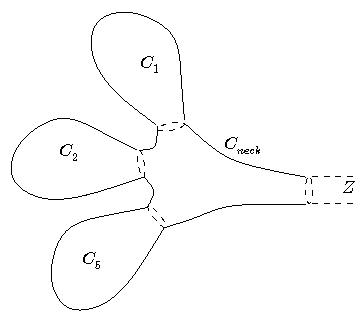} 
     \caption{Gluing three curves to a cylinder. }
     \label{fig:1}
  \end{figure}
  
\begin{example}\label{ex:1}\rm (i)  In order to calculate $c_k(7)$ by our method,
we must find elements in $\Mm(d,7+\eps,t)$ for suitable $d,t$ and small $\eps>0$.  
The index condition  $0 =  3d-1 - t- \lfloor \frac t{7+\eps}\rfloor$ has  no solution with $t=7$.
  However, by Remark~\ref{rmk:stab}, any element
 in $\Mm(8,7+\eps, 21)$ also stabilizes.  
   Hence it would suffice to show that this moduli space is nonempty.
   
One might try to construct an element in $\Mm(8,7+\eps, 21)$ by gluing three curves $C_1,C_2, C_5$ (where $C_d$ has degree $d$) 
   at $y= \frac{13}2 + \eps$ to a cylinder $$
   Z\in \Mm_{cyl}\bigl(\be_1^{20}(y), \be_1^{21}(7+\eps)\bigr)
   $$
    as in Fig.~\ref{fig:1}.
   We have $C_1 \in \Mm(1,y,2)$ and $C_2\in \Mm(2,y,5)$  by Lemma~\ref{le:3k-1},  and also
 $\Mm(5, y, 13)\ne \emptyset$ since $2,5,13$ are odd  index Fibonacci numbers (for a proof, see \cite{CGH}).   Further, the partition $(2,5,13)$ of  $20$ satisfies the conditions of Lemma~\ref{le:2} at $y = 
 \frac{13}2 + \eps$.   
 Hence if there were a cylinder $Z\in \Mm_{cyl}\bigl(\be_1^{20}(y), \be_1^{21}(7+\eps)\bigr)$, we could glue this cylinder to 
 $C_1,C_2,C_5$   to obtain an element in $\Mm(8,7+\eps, 21)$.  
But  there is a multilevel building from $\be_1^{20}(y)$  to $\be_1^{21}(7+\eps)$ that consists of a cylinder from  
$\be_1^{20}(y)$ to $\be_2^{3}(y)$, together with  the $3$-fold cover of 
a cylinder from $\be_2(y)$ to $\be_1^7(7+\eps)$.  (Note that both these ends have $\frac 12$-index equal to $7$.)
Hence there is no reason that there should also be a cylinder $Z$ with these ends, and
in fact one can use ideas from ECH to show that such a cylinder $Z$ cannot exist, as follows.

Such a cylinder $Z$ would have to be somewhere injective, 
and hence have nonnegative ECH index $I(Z)$.\footnote{
See \cite{Hlect} or \cite[\S2]{CGHM} for relevant background on ECH.}
  Moreover, if it had zero ECH index, it would have to have ECH partitions at its ends.
But its ECH index is $\gr\bigl(\be_1^{20}(y)\bigr) - \gr\bigl(\be_1^{21}(x)\bigr)$,
where $1+ \frac 12 \gr\bigl(\be_1^{p}(z)\bigr)$ is the number of integer points in the triangle with vertices $(0,0)$, $(p,0)$ and $(0, \frac pz)$.\footnote
{
This somewhat surprising interpretation of the ECH index holds because the latter is a sum of Conley--Zehnder Fredholm-type indices that as above have the form $\lfloor zk \rfloor$ and hence are related to numbers of integer points: see \cite{Hlect}.}  
An easy calculation shows that 
 $$
\tfrac 12 \gr\bigl(\be_1^{20}(y)\bigr) = \tfrac 12\gr\bigl(\be_1^{21}(x)\bigr) = 42.
$$
Thus the ECH index $I(Z)$ of $Z$ would be  zero.  Now the ECH
 partition
  $p^-_{1/(7+\eps)}(21)$ at the negative end 
 is given by the $x_1$-lengths of the segments of the minimal convex lattice path (i.e. with integral break points) that lies above the line $x_2 = \frac 1{7+\eps} x_1$.  Hence because $\frac{21}{7+\eps}$ is just less than $3$, this partition is $p^-_{1/(7+\eps)}(21)=(7^{\times 3})$.  Thus $Z$ does not have ECH partitions, and so cannot exist.
 
 In the published version of this paper, an alternative discussion of why the cylinder $Z$  cannot exist was based on the assumption that the three curves $C_1,C_2, C_5$ can be glued at $y$ to give an element in $\Mm(8,y, 20)$.  As explained in Remark~\ref{rmk:trivcyl} above, this gluing operation does not exist.  Further one can use  Hutchings' adjunction formula
 (see \cite{Hlect}, or  \cite[eq~(2.2.10),~Remark~2.2.3]{CGHM})
 to show that $\Mm(8,y, 20)$ is empty.   Indeed this formula shows that the union $C = C_1\cup C_2\cup C_5$ has one double point, occurring as the intersection of $C_1$ with $C_5$.  (As in \cite[Remark~3.1.8~(ii)]{CGHM}, we may obtain the $C_i$ from curves in classes $L-E_{12},\; 2L-E_{1\dots5}, \;5L-2E_{1\dots 6}-E_{78}$ by stretching the neck; note that  $(L-E_{12})\cdot (5L-2E_{1\dots 6}-E_{78}) = 1$.)
  However, replacing the partition $(2,5,13)$ given by $C$ with the partition $(20)$ increases the writhe at the negative end, giving a negative value to the number of double points predicted by the adjunction formula.  
\MS

\NI (ii)  Now let us consider the point $x = \frac {76}{11} + \eps$.\footnote{
Despite appearances, this example is not completely random. Indeed, $ \frac {76}{11}  = 7-\frac 1{11}$ bears the same relation to $11$ as the Fibonacci quotient $\frac {55}8$ does to $8$: i.e.  in the notation of \cite[Lemma~4.1.2]{MS} we have $\frac {76}{11}= v_1(10)$ 
while $\frac {55}{8}= v_1(7)$.  Thus $\frac {76}{11}$ is the first in the next set of generalized ghost stairs.
}
It is not hard to check that elements in $\Mm(29, x, 76)$ do stabilize.   Further, the adjunction formula does not rule out the existence of a curve here, and one might hope to 
 construct an element of $\Mm(29, x, 76)$ by gluing
  suitable curves in $\ov X_{y}$ for some $y<x$  to a cylinder $Z$ from $\be_1^{75}(y)$ to $\be_1^{76}(x)$.  For example, if $y = \frac{34}5+\eps$, then there are elements in $C_{13}\in \Mm(13,y, 34)$,
$C_2\in \Mm(2,y, 5)$ and $C_1\in \Mm(1,y, 2)$ and if one could  glue two copies of
$C_{13}$ to $ C_2$ and $C_1$
then one would get a curve of  the right degree, namely $29$.
However, even if $Z$ existed, 
%
%
Proposition~\ref{prop:HT} does not permit such a gluing  since that requires
the curves $u_i$ to be distinct, and the adjunction formula again shows that one cannot hope to glue in stages,
first constructing an element in $\Mm(16, y, 41)$ by gluing one copy of each of $C_{13}, C_2$ and $C_1$ to a trivial cylinder, and then gluing in the second copy of $C_{13}$. 

Another difficulty with this approach is that there is no cylinder $Z$ with ends on $\be_1^{75}(y)$ and $\be_1^{76}(x)$.  This follows  by an easy ECH calculation.  Such a curve would have to be somewhere injective, so that we would have to have
$\gr\bigl(\be_1^{75}(y)\bigr) \ge \gr\bigl(\be_1^{76}(x)\bigr)$. 
But one can calculate that
$
\tfrac 12 \gr\bigl(\be_1^{75}(y)\bigr) = 456,\ \  \tfrac 12 \gr\bigl(\be_1^{76}(x)\bigr) = 461.
$
So $Z$ cannot exist.

Nevertheless, 
since the orbits $\be_1^{75}(y)$ and $\be_1^{76}(x)$ have the same Fredholm index we saw above that there is some cylindrical building between them.    There are many possible candidates.  For example,  note that both
 $\be_2(y)$ and $\be_2(x)$ 
 have index $7$ so that there could be a cylinder $Z_0$ of index $0$ between them, but
\begin{align*}
&\tfrac 12 \ind(\be_2^{11}(y)) = 11 + 74 < \tfrac 12 \ind(\be_1^{75}(y)) = 75 + 11, \mbox { while }\\
&\tfrac 12 \ind(\be_2^{11}(x)) = 11 + 76 > \tfrac 12 \ind(\be_1^{76}(x)) = 76 + 10.
\end{align*}
Hence there might be a building 
with three levels: first a cylinder in $\p E(1,y)\times \R$
 that goes from $ \be_1^{75}(y)$ to $\be_2^{11}(y)$ in $\p E(1,y)\times \R$, then across the cobordism via $11$ times $Z_0$,  and then from $\be_2^{11}(x)$ to $\be_1^{76}(x)$ in $\p E(1,x)\times \R$: see  
Figure~\ref{fig:2}. 
 \begin{figure}[htbp] 
   \centering
   \includegraphics[width=3in]{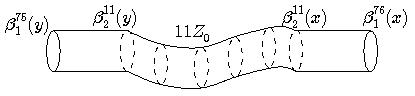} 
   \caption{A candidate multi-level building from $\be_1^{75}(y)$ to $\be_1^{76}(x)$}
   \label{fig:2}
\end{figure}
In order to glue such a building to the curves $C_1, C_2, C_{13}$ mentioned above we would need a significantly
stronger version of the Hutchings--Taubes gluing theorem.  
\hfill$\er$
\end{example}  

The next remark explains some cases in which the gluing coefficient has to vanish because otherwise the glued curve would violate some ECH identity such as the adjunction formula.  It is not known whether these numerical conditions are sufficiently strong to detect all such cases.

\begin{rmk}\label{rmk:glu}{\bf (More about gluing)\;} \rm  
(i)  Here is a very easy example in which a given building cannot be glued to a single curve. Let $y\in (5,7)$, and
take the curves $C_i \in \Mm(i, y, t_i)$ for $i=1,2$  in Example~\ref{ex:1} (i) (where  $t_1=2, t_2=5$) together with a 
genus zero curve in  the \lq neck' $\R\times \p E(1,y)$ with two positive ends of multiplicities $(2,5)$ and one negative end of multiplicity $7$.
If these could be glued at the parameter value $y$, we would get a curve in $\Mm(3,y,7)$ of index $0$.  But this cannot exist because if it did 
the adjunction formula would imply that it has a negative number of double points.  (As in Example~\ref{ex:1}, the curves $C_1,C_2$ are disjoint, and the writhe of a negative end of multiplicity $7$ is greater than than one with partition $(2,5)$ for any $y$ in this range.)
 If $5< y <6$, one can use an easier argument: because $\tfrac 12 \gr(\be_1^7(y)) = 9$ in this range, this glued curve would have ECH index $=0$ and hence
the ECH partition $p^-_{1/y}(7)$ at its bottom end. But in this range $p^-_{1/y}(7) = (5,2)$.
 
 \MS
 
\NI (ii) More generally, one might try to glue $C_1, C_2$ at $y\in (5,7)$ to a multiply covered cylinder from $\be_1^7(y)$ to $\be_1^7(x)$ for some $7/2< x < 5$ (where the lower bound is chosen so that $ \ind(\be_1^7(x)) =  \ind(\be_1^7(y))$.
 Such a gluing operation should be possible 
in the context used in cylindrical contact homology, but for this, as explained in the work of Hutchings--Nelson,  one must use  domain dependent $J$; more precisely one sets up the theory so that $J$ depends on the circle coordinate in the domain.  However, 
  although one probably can do the gluing for domain dependent $J$ and hence obtain a curve in $\ov X_x$ for $x> 7/2$,  because we can no longer rely on automatic regularity~\cite{W} 
there is no reason why  the resulting glued curves cannot be cancelled by some others   
 as one perturbs $J$ to make it domain independent.  Thus
 there is no reason why  there should be a limiting building for  domain independent $J$.\footnote
 {
 Note that we cannot assert the existence of such a building by means of  a cobordism argument in cylindrical contact homology because the top end is not cylindrical.  Indeed,  when one removes a line from $\ov X_y$ and completes the resulting manifold at its positive end, the curve $C$  becomes asymptotic to a curve $C'$  with three ends each of multiplicity $1$, rather than one end of multiplicity $3$.}
 This is just as well, since when 
$ \frac 72< x < 5$  there is no suitable limiting   genus zero $J$-holomorphic building.  This would have to have a top level of degree $3$ in $ \ov X_x$ (with $\frac12$-grading $9$),   bottom levels in the symplectization $\p E(1,x)\times \R$, and
 bottom end asymptotic to $\be_1^7(x)$ (with $\frac12$-grading $10$ for  $x\in (4,5)$ and $11$ for $x\in (\frac 72,4)$).   Thus the total ECH index would be  negative.    Since somewhere injective curves always have nonnegative ECH index,  at least one of its components must therefore be multiply covered.  But it is easy to check that  no buildings satisfy these requirements. 
 For example, if the top  component in $ \ov X_x$ were a triply covered copy of $C_1$ it would have end of multiplicity $6$, and there is no holomorphic curve with top end on $\be_1^6(x)$ and bottom end on $\be_1^7(x)$.
\hfill$\er$
\end{rmk}

\NI {\bf Acknowledgements.}  This paper is an offshoot of my joint project with Dan Cristofaro-Gardiner and Richard Hind to calculate $c_k(x)$ for all $x>\tau^4$.  I wish to thank them for being such wonderful mathematicians to work with,  Michael Hutchings for help with understanding obstruction bundle gluing,   the referee whose careful comments helped clarify the exposition, and finally Dan Cristofaro-Gardiner  for a very thorough reading of this revision.   Also I wish to recognize Dr. Edge who taught me when I was an undergraduate at Edinburgh University in 1962-67.  He gave an inspiring   course on projective geometry: I still remember him telling us that the secret of the $27$ lines on the cubic surface  lay in four dimensions -- very mysterious to me at the time!

\end{document}